\documentclass[11pt]{article}
\usepackage{hyperref}
\usepackage{times}  
\usepackage{mathpazo}
\usepackage{amssymb,amsmath,amsthm}
\usepackage{epsfig}

\newtheorem{theorem}{Theorem}[section]
\newtheorem{proposition}[theorem]{Proposition}
\newtheorem{definition}[theorem]{Definition}

\newtheorem{claim}[theorem]{Claim}
\newtheorem{lemma}[theorem]{Lemma}

\newcommand{\qedsymb}{\hfill{\rule{2mm}{2mm}}}
\renewenvironment{proof}[1][]{\begin{trivlist}
\item[\hspace{\labelsep}{\bf\noindent Proof#1:\/}] }{\qedsymb\end{trivlist}}

\def\calC{{\cal C}}

\def\R{\mathbb{R}}

\def\mod{\mbox{mod}}

\def\bp{\mathop{\mathrm{bp}}}

\def\supp{\mathop{\mathrm{supp}}}
\def\disc{\mathop{\mathrm{disc}}}
\def\ent{\mbox{H}_\infty}

\newcommand\Prob[2]{{\Pr_{#1}\left[ {#2} \right]}}

\renewcommand{\epsilon}{\varepsilon}

\newcommand{\rank}{\mathop{\mathrm{rank}}}
\newcommand{\Rbool}{{\rank}_\mathbb{B}}
\newcommand{\Rbin}{{\rank}_{\mathrm{bin}}}
\newcommand{\Rreal}{{\rank}_\mathbb{R}}

\begin{document}

\title{{\bf On the Binary and Boolean Rank of Regular Matrices}
}

\author{
Ishay Haviv\thanks{School of Computer Science, The Academic College of Tel Aviv-Yaffo, Tel Aviv 61083, Israel. Research supported in part by the Israel Science Foundation (grant No.~1218/20).}
\and
Michal Parnas\thanks{School of Computer Science, The Academic College of Tel Aviv-Yaffo, Tel Aviv 61083, Israel. Email address: {\tt michalp@mta.ac.il}}
}

\date{}

\maketitle

\begin{abstract}
A $0,1$ matrix is said to be regular if all of its rows and columns have the same number of ones.
We prove that for infinitely many integers $k$, there exists a square regular $0,1$ matrix with binary rank $k$, such that the Boolean rank of its complement is $k^{\widetilde{\Omega}(\log k)}$.
Equivalently, the ones in the matrix can be partitioned into $k$ combinatorial rectangles, whereas the number of rectangles needed for any cover of its zeros is $k^{\widetilde{\Omega}(\log k)}$.
This settles, in a strong form, a question of Pullman (Linear Algebra Appl.,~1988) and a conjecture of Hefner, Henson, Lundgren, and Maybee (Congr. Numer.,~1990).
The result can be viewed as a regular analogue of a recent result of Balodis, Ben-David, G{\"{o}}{\"{o}}s, Jain, and Kothari (FOCS,~2021), motivated by the clique vs. independent set problem in communication complexity and by the (disproved) Alon-Saks-Seymour conjecture in graph theory.
As an application of the produced regular matrices, we obtain regular counterexamples to the Alon-Saks-Seymour conjecture and prove that for infinitely many integers $k$, there exists a regular graph with biclique partition number $k$ and chromatic number $k^{\widetilde{\Omega}(\log k)}$.
\end{abstract}

\section{Introduction}

For a $0,1$ matrix $M$ of dimensions $n \times m$, consider the following three notions of rank.
\begin{itemize}
  \item The (standard) {\em rank} of $M$ over $\R$, denoted by $\Rreal(M)$, is the minimal $k$ for which there exist real matrices $A$ and $B$ of dimensions $n \times k$ and $k \times m$ respectively, such that $M = A \cdot B$ where the operations are over $\R$.
  \item The {\em binary rank} of $M$, denoted by $\Rbin(M)$, is the minimal $k$ for which there exist $0,1$ matrices $A$ and $B$ of dimensions $n \times k$ and $k \times m$ respectively, such that $M = A \cdot B$ where the operations are over $\R$. Equivalently, $\Rbin(M)$ is the smallest number of monochromatic combinatorial rectangles\footnote{A {\em (combinatorial) rectangle} in a matrix $M$ is a set $P \times Q$, where $P$ and $Q$ are sets of rows and columns in $M$ respectively. The rectangle is said to be {\em monochromatic} if the entries of the corresponding submatrix of $M$ are all equal.} in a {\em partition} of the ones in $M$.
  \item The {\em Boolean rank} of $M$, denoted by $\Rbool(M)$, is the minimal $k$ for which there exist $0,1$ matrices $A$ and $B$ of dimensions $n \times k$ and $k \times m$ respectively, such that $M = A \cdot B$ where the operations are under Boolean arithmetic (namely, $0+x=x+0=x$, $1+1=1 \cdot 1 = 1$, and $x \cdot 0 = 0 \cdot x = 0$). Equivalently, $\Rbool(M)$ is the smallest number of monochromatic combinatorial rectangles in a {\em cover} of the ones in $M$.
\end{itemize}
Note that the binary rank and the Boolean rank are sometimes referred to in the literature as the $1$-partition number and the $1$-cover number respectively.
Note further that every $0,1$ matrix $M$ satisfies $\Rbin(M) \geq \Rreal(M)$ and $\Rbin(M) \geq \Rbool(M)$.

The above notions of rank play a central role in the area of communication complexity, introduced in 1979 by Yao~\cite{Yao79}.
In the communication problem associated with a $0,1$ matrix $M$ of dimensions $n \times m$, one player holds a row index $i \in [n]$ and another player holds a column index $j \in [m]$, and their goal is to decide whether $M_{i,j}=1$ while minimizing the worst-case number of communicated bits.
For the deterministic setting, the well-known log-rank conjecture of Lov{\'a}sz and Saks~\cite{LovaszS88} suggests that the communication complexity of the problem is polynomially related to $\log_2 \Rreal(M)$ (see, e.g.,~\cite{LovettA14}).
For the non-deterministic setting, it is not difficult to see that the minimum number of bits that should be communicated is precisely $\lceil \log_2 \Rbool(M) \rceil$.
For the {\em unambiguous} non-deterministic setting, where each input is required to have at most one accepting computation, the minimum number of bits that should be communicated is precisely $\lceil \log_2 \Rbin(M) \rceil$.

For a $0,1$ matrix $M$, let $\overline{M}$ denote the complement matrix obtained from $M$ by replacing the ones by zeros and the zeros by ones.
A result of Yannakakis~\cite{Yannakakis91} implies that every $0,1$ matrix $M$ with $\Rbin(M) = k$ satisfies
\begin{eqnarray}\label{eq:M_upper}
\Rbool( \overline{M}) \leq \Rbin( \overline{M}) \leq k^{O(\log k)}.
\end{eqnarray}
The challenge of determining the largest possible value of $\Rbool( \overline{M})$ for a $0,1$ matrix $M$ with $\Rbin(M) = k$ has attracted intensive attention in the literature, mostly with the equivalent formulation of the clique vs. independent set problem introduced in~\cite{Yannakakis91} (see~\cite[Chapter~4.4]{JuknaBook}).
The first non-trivial bound was given by Huang and Sudakov~\cite{HuangS12} who provided, building on a construction of Razborov~\cite{Razborov92}, a family of such matrices $M$ satisfying $\Rbool( \overline{M}) \geq \Omega(k^{6/5})$ (see~\cite{CioabaT11} for extended constructions).
The constant $6/5$ in the exponent was improved to $3/2$ by Amano~\cite{Amano14} and then to $2$ by Shigeta and Amano~\cite{ShigetaA15}.
The first super-polynomial separation was obtained by G{\"{o}}{\"{o}}s~\cite{Goos15}, who provided a family of such matrices $M$ satisfying $\Rbool( \overline{M}) \geq k^{\Omega(\log^{0.128} k)}$. This was then improved in a work of Ben-David, Hatami, and Tal~\cite{BHT17} to $\Rbool( \overline{M}) \geq k^{\Omega(\log^{0.22} k)}$. In a recent breakthrough, it was shown by Balodis, Ben-David, G{\"{o}}{\"{o}}s, Jain, and Kothari~\cite{BBGJK21} that the bound can be further improved to $\Rbool( \overline{M}) \geq k^{\widetilde{\Omega}(\log k)}$, which matches the upper bound in~\eqref{eq:M_upper} up to $\log \log k$ factors hidden in the $\widetilde{\Omega}$ notation.
Note that the result of~\cite{BBGJK21} strengthens an earlier result of G{\"{o}}{\"{o}}s, Pitassi, and Watson~\cite{GoosP018}, who provided a near optimal separation between the binary rank of a $0,1$ matrix and the deterministic communication complexity of the problem associated with it.

Interestingly, the above problem is closely related to a graph-theoretic problem proposed by Alon, Saks, and Seymour in 1991 (see~\cite{Kahn91}).
For a graph $G$, let $\chi(G)$ denote its chromatic number, and let $\bp(G)$ denote its biclique partition number, that is, the smallest number of edge-disjoint bicliques (i.e., complete bipartite graphs) needed for a partition of the edge set of $G$.
A classic result of Graham and Pollak~\cite{GrahamP71} asserts that the complete graph $K_n$ on $n$ vertices satisfies $\bp(K_n) = n-1$. Inspired by this result, Alon, Saks, and Seymour conjectured that every graph $G$ satisfies $\bp(G) \geq \chi(G) - 1$. The conjecture was disproved by Huang and Sudakov in~\cite{HuangS12}, where it was shown that for infinitely many integers $k$ there exists a graph $G$ satisfying $\bp(G)=k$ and $\chi(G) \geq \Omega(k^{6/5})$. These graphs were used there to derive the aforementioned separation between $\Rbin(M)$ and $\Rbool( \overline{M})$ for $0,1$ matrices $M$ (see~\cite[Section~4]{HuangS12}).
In a work of Bousquet, Lagoutte, and Thomass{\'{e}}~\cite{BousquetLT14}, the two problems were shown to be essentially equivalent, allowing the authors of~\cite{BBGJK21} to derive, for infinitely many integers $k$, the existence of a graph $G$ satisfying $\bp(G)=k$ and $\chi(G) \geq k^{\widetilde{\Omega}(\log k)}$.
As in the matrix setting, the gap is optimal up to $\log \log k$ factors in the exponent.

A $0,1$ matrix $M$ is said to be {\em $d$-regular} if every row and every column in $M$ has precisely $d$ ones.
In 1986, Brualdi, Manber, and Ross~\cite{BrualdiMR86} proved that for every $d$-regular $0,1$ matrix $M$ of dimensions $n \times n$ where $0 < d < n$, the rank of $M$ over the reals is equal to that of its complement, that is, $\Rreal(M) = \Rreal(\overline{M})$.
Following their work, Pullman~\cite{Pullman88} asked in 1988 whether every such matrix $M$ satisfies $\Rbin(M) = \Rbin(\overline{M})$. In 1990, Hefner, Henson, Lundgren, and Maybee~\cite{HefnerHLM90} conjectured that the answer to this question is negative (see~\cite[Conjecture~3.2]{HefnerHLM90}).
The question was asked again in 1995 in a survey by Monson, Pullman, and Rees~\cite{MonsonPR95survey} (see~\cite[Open problem~7.1]{MonsonPR95survey}).\footnote{The question of~\cite{Pullman88,HefnerHLM90,MonsonPR95survey} was originally formulated using the notion of {\em non-negative integer rank}, which coincides with the binary rank for $0,1$ matrices (see, e.g.,~\cite[Lemma~2.1]{GregoryPJL91}).}
Note that for the Boolean rank, such a statement does not hold in general. For example, the $1$-regular identity matrix $I_n$ satisfies $\Rbool(I_n) = n$ and yet $\Rbool(\overline{I_n}) = (1+o(1)) \cdot \log_2 n$ (see~\cite{deCaenGP81}).

\subsection{Our Contribution}

The current work aims to determine the largest possible gap between the binary rank of {\em regular} $0,1$ matrices and the Boolean rank of their complement.
Our main result is the following.

\begin{theorem}\label{thm:main}
For infinitely many integers $k$, there exists a square regular $0,1$ matrix $M$ satisfying
\[\Rbin(M) = k \mbox{~~~and~~~} \Rbool(\overline{M}) \geq k^{\widetilde{\Omega}(\log k)}.\]
\end{theorem}
\noindent
Theorem~\ref{thm:main} can be viewed as a regular analogue of the aforementioned result of Balodis et al.~\cite{BBGJK21}, showing that their near optimal separation between $\Rbin(M)$ and $\Rbool(\overline{M})$, which is achieved by irregular matrices $M$, can also be attained by regular ones.
Since every $0,1$ matrix $M$ satisfies $\Rbin(\overline{M}) \geq \Rbool(\overline{M})$, Theorem~\ref{thm:main} settles, in a strong form, the question of Pullman asked in~\cite{Pullman88,MonsonPR95survey} (and the variants of the question mentioned there) and confirms the conjecture of Hefner et al.~\cite{HefnerHLM90}.
We remark that regular matrices $M$ with $\Rbool(\overline{M})$ larger than $\Rbin(M)$ can also be derived from~\cite{HuangS12} (see Section~\ref{sec:overview} for details). While these matrices are sufficient to answer the original question of~\cite{Pullman88,MonsonPR95survey}, they only achieve a polynomial gap between the quantities.

The proof of Theorem~\ref{thm:main} relies on a modification of the construction of~\cite{BBGJK21} to the regular setting.
It involves an extension of the query-to-communication lifting theorem in non-deterministic communication complexity proved by G{\"{o}}{\"{o}}s, Lovett, Meka, Watson, and Zuckerman~\cite{GoosLMWZ16}, as well as a two-source extractor studied by Bouda, Pivoluska, and Plesch~\cite{BoudaPP12} and by Kothari, Meka, and Raghavendra~\cite{KothariMR17}. For an overview of the proof, see Section~\ref{sec:overview}.

As alluded to before, matrices $M$ with $\Rbin(M)$ much smaller than $\Rbool(\overline{M})$ are known to imply graphs $G$ with $\bp(G)$ much smaller than $\chi(G)$,
and thus yield counterexamples to the Alon-Saks-Seymour conjecture (see~\cite{BousquetLT14}).
Although the conjecture is false in general, it is of interest to identify classes of graphs that satisfy a polynomial version of the conjecture.
In particular, it was asked in~\cite{BBGJK21} whether the chromatic number of perfect graphs is polynomially upper bounded in terms of their biclique partition number (see~\cite{Yannakakis91} for a related question; see also~\cite{LagoutteT16,BousquetLMP18,CS21}).
As an application of Theorem~\ref{thm:main}, we show that this is not the case for the class of regular graphs.
Namely, we show that the near optimal separation achieved in~\cite{BBGJK21} between the biclique partition number and the chromatic number can also be attained by regular graphs.

\begin{theorem}\label{thm:ASS-regular}
For infinitely many integers $k$, there exists a simple regular graph $G$ satisfying
\[{\bp}(G) = k \mbox{~~~and~~~} \chi(G) \geq k^{\widetilde{\Omega}(\log k)}.\]
\end{theorem}

\subsection{Overview of Proofs}\label{sec:overview}

Our goal is to obtain regular $0,1$ matrices $M$ for which the binary rank of $M$ is much smaller than the Boolean rank of $\overline{M}$.
We first observe that a polynomial gap between the two quantities, for a regular matrix, can be derived from a construction of Huang and Sudakov~\cite{HuangS12}.
Indeed, it can be verified that the (simple) graphs $G$ given in~\cite{HuangS12}, which satisfy $\bp(G) = k$ and $\chi(G) \geq \Omega(k^{6/5})$, are regular, hence their adjacency matrices are regular as well. The following simple claim implies that these adjacency matrices achieve a polynomial gap between the binary rank and the Boolean rank of the complement.
\begin{claim}
For every simple graph $G$, the adjacency matrix $M$ of $G$ satisfies
\[\Rbin(M) \leq 2 \cdot \bp(G)~~~\mbox{and}~~~\Rbool(\overline{M}) \geq \chi(G).\]
\end{claim}
\begin{proof}
For a simple graph $G$ on the vertex set $[n]$, put $k = \bp(G)$, and let $(A_1,B_1), \ldots, (A_k,B_k)$ be the $k$ bipartitions of the $k$ edge-disjoint bicliques that form a partition of the edge set of $G$. Observe that for every $i \in [k]$, the sets $A_i \times B_i$ and $B_i \times A_i$ form combinatorial rectangles of ones in the adjacency matrix $M$ of $G$, and that these $2k$ rectangles form a partition of the ones in $M$, hence $\Rbin(M) \leq 2 \cdot k$.

Next, put $m = \Rbool(\overline{M})$, and let $A_1 \times B_1, \ldots, A_m \times B_m$ be $m$ combinatorial rectangles that form a cover of the ones in $\overline{M}$, i.e., the zeros in $M$. For every $i \in [m]$, let $C_i$ denote the set of elements $j \in [n]$ satisfying $(j,j) \in A_i \times B_i$. Since $G$ is simple, the elements on the diagonal of $M$ are all zeros, hence the sets $C_i$ for $i \in [m]$ cover all vertices of $G$. Since $A_i \times B_i$ is a rectangle of zeros in $M$,  it also follows that $C_i$ is an independent set in $G$. This implies that $m \geq \chi(G)$, and we are done.
\end{proof}

The matrices $M$ that are known to achieve super-polynomial separations between $\Rbin(M)$ and $\Rbool(\overline{M})$, however, are apparently far from being regular~\cite{Goos15,BHT17,BBGJK21}.
Their constructions rely on a powerful technique, known as {\em query-to-communication lifting}, that enables to deduce separation results in communication complexity from separation results in the more approachable area of query complexity.
The proofs of the separation results of~\cite{Goos15,BHT17,BBGJK21} involve two main steps, as described below.

In the first step, one provides a family of Boolean functions $f: \{0,1\}^n \rightarrow \{0,1\}$ with a large gap between two certain measures of Boolean functions, namely, the unambiguous $1$-certificate complexity of $f$ and the $0$-certificate complexity of $f$ (see Section~\ref{sec:query_def}).
These measures can be viewed as query complexity analogues of the binary rank of a matrix and the Boolean rank of its complement.
It is shown in~\cite{BBGJK21} that the gap between the two measures can be nearly quadratic.

In the second step, the separation is ``lifted'' from query complexity to communication complexity. This is done by considering, for some gadget function $g: \{0,1\}^\ell \times \{0,1\}^\ell \rightarrow \{0,1\}$, the communication problem in which two players get inputs from $\{0,1\}^{\ell \cdot n}$ and aim to determine the value of the composed function $f \circ g^n : \{0,1\}^{\ell \cdot n} \times \{0,1\}^{\ell \cdot n} \rightarrow \{0,1\}$, defined by
\[ (f \circ g^n) (x,y) = f(g(x_1,y_1), g(x_2,y_2), \ldots, g(x_n,y_n))\]
for all $x,y \in \{0,1\}^{\ell \cdot n}$. Here, the vectors $x=(x_1, \ldots, x_n)$ and $y = (y_1, \ldots, y_n)$ are viewed as concatenations of $n$ blocks of size $\ell$.
Query-to-communication lifting results typically show that for some gadget $g$, a gap between certain query complexity measures of $f$ implies a gap between the suitable communication complexity measures of the composed function $f \circ g^n$. For the non-deterministic setting, it is shown in~\cite{GoosLMWZ16} that if the gadget $g$ is the inner product function on vectors of length $\ell = \Theta(\log n)$, then a gap between the unambiguous $1$-certificate complexity and the $0$-certificate complexity for $f$ implies a gap between the unambiguous non-deterministic communication complexity and the co-non-deterministic communication complexity for $f \circ g^n$ (see also~\cite[Appendix~A]{Goos15}).
The analysis uses the fact that the inner product function forms a two-source extractor, as shown by Chor and Goldreich~\cite{ChorG88}.

Let $M$ denote the matrix associated with the communication problem of $f \circ g^n$ for the function $f$ constructed in~\cite{BBGJK21} and the inner product function $g$.
The lifting result of~\cite{GoosLMWZ16} implies that $M$ attains a near optimal separation between $\Rbin(M)$ and $\Rbool(\overline{M})$.
However, it can be seen that the matrix $M$ is not regular at all. For example, the row and the column of $M$ that correspond to the all-zero vector consist of only ones or only zeros, depending on the value of $f$ on the all-zero vector.

We turn to describe how we obtain regular matrices $M$ with a similar gap between $\Rbin(M)$ and $\Rbool(\overline{M})$.
We first observe that to construct a regular matrix $M$, it suffices to replace the inner product function in the above construction by a different gadget function $g$. Specifically, it turns out that if $g$ is unbiased in a strong sense, namely, it is unbiased even while fixing one of its two inputs, then the matrix $M$ associated with $f \circ g^n$ is regular for any function $f$ (see Section~\ref{sec:unbiased}).
Hence, to obtain the desired separation on regular matrices, we provide an extension of the query-to-communication lifting theorem of~\cite{GoosLMWZ16} which allows the gadget function $g$ to be not only the inner product function but any low-discrepancy function. We note that such an extension was speculated already in~\cite[Remark~1]{GoosLMWZ16} and was actually established for the deterministic and probabilistic settings in a recent work of Chattopadhyay, Filmus, Koroth, Meir, and Pitassi~\cite{ChattopadhyayFK21}. Building on the approach of~\cite{GoosLMWZ16} and on tools supplied in~\cite{ChattopadhyayFK21}, we prove that such an extension holds for the non-deterministic setting as well (for a precise statement, see Theorem~\ref{thm:lift}).
We proceed by showing that a slight variant $g$ of the inner product function, studied in~\cite{BoudaPP12} and in~\cite{KothariMR17}, is unbiased in the required sense and has low discrepancy.
Then, to prove Theorem~\ref{thm:main}, we apply our generalized query-to-communication lifting theorem to the family of functions $f$ provided in~\cite{BBGJK21} with this gadget $g$.

Let us mention that our generalized lifting theorem is not essential for the proof of Theorem~\ref{thm:main}. It turns out that the matrix $M$ obtained using the aforementioned gadget function $g$ has a sub-matrix that corresponds to a composition with the standard inner product function, hence the lower bound on $\Rbool(\overline{M})$ can also be derived from the lifting result of~\cite{GoosLMWZ16}. Yet, the generality of our lifting theorem, proved in Appendix~\ref{appendix:lifting}, can be used to obtain a separation between $\Rbin(M)$ and $\Rbool(\overline{M})$ using various other gadget functions, and we believe that it might find additional applications.

We finally use the regular matrices given in Theorem~\ref{thm:main} to provide regular counterexamples for the Alon-Saks-Seymour conjecture and to prove Theorem~\ref{thm:ASS-regular}.
It is shown in~\cite{BousquetLT14} that a matrix $M$ with $\Rbin(M)$ much smaller than $\Rbool(\overline{M})$ can be transformed into a graph $G$ with $\bp(G)$ much smaller than $\chi(G)$. This transformation, however, does not preserve the regularity.
In fact, a natural attempt to produce a regular graph $G$ from a regular matrix $M$ using the approach of~\cite{BousquetLT14} results in a graph that is not even simple (because it has loops).
Moreover, certain steps of the argument of~\cite{BousquetLT14} identify subgraphs of this graph $G$ with a biclique partition number much smaller than the chromatic number, but those subgraphs are not necessarily regular even if $G$ is.
We overcome these difficulties by combining the approach of~\cite{BousquetLT14} with a couple of additional ideas, and show that any square regular matrix $M$ with a large gap between $\Rbin(M)$ and $\Rbool(\overline{M})$ can be transformed into a simple regular graph $G$ with a similar gap between $\bp(G)$ and $\chi(G)$ (see Theorem~\ref{thm:M->G}).

\subsection{Outline}
The rest of the paper is organized as follows.
In Section~\ref{sec:preliminaries}, we collect several definitions and results needed throughout the paper.
In Section~\ref{sec:lifting}, we present our generalized query-to-communication lifting theorem in non-deterministic communication complexity.
Its proof is given in Appendix~\ref{appendix:lifting}.
In Section~\ref{sec:matrices}, we present and analyze a certain gadget function, and combine it with the lifting theorem to prove Theorem~\ref{thm:main}.
Finally, in Section~\ref{sec:graphs}, we obtain regular graphs that form counterexamples to the Alon-Saks-Seymour conjecture and confirm Theorem~\ref{thm:ASS-regular}.

\section{Preliminaries}\label{sec:preliminaries}

\subsection{Non-deterministic Communication Complexity}

Let $\Lambda$ be a finite set, and let $F: \Lambda \times \Lambda \rightarrow \{0,1\}$ be a function.
In the communication problem associated with $F$, one player holds an input $x \in \Lambda$ and another player holds an input $y \in \Lambda$, and their goal is to decide whether $F(x,y)=1$ by a communication protocol that minimizes the worst-case number of communicated bits.
The $0,1$ matrix $M$ associated with the function $F$ is the matrix whose rows and columns are indexed by $\Lambda$, defined by $M_{x,y}=F(x,y)$ for all $x,y \in\Lambda$.
Consider the following three non-deterministic communication complexity measures of a function $F$.
\begin{itemize}
  \item The {\em non-deterministic communication complexity} of $F$, denoted by $\mathsf{NP^{cc}}(F)$, is the smallest possible number of communicated bits in a non-deterministic communication protocol for $F$, that is, a protocol satisfying that $F(x,y)=1$ if and only if there exists an accepting computation on $(x,y)$.
It holds that $\mathsf{NP^{cc}} (F) = \lceil \log_2 \Rbool(M) \rceil$.
  \item The {\em co-non-deterministic communication complexity} of $F$, denoted by $\mathsf{coNP^{cc}}(F)$, is the non-deterministic communication complexity of the negation $\neg F$ of $F$, defined by $(\neg F)(x,y) = 1-F(x,y)$ for all $x,y \in\Lambda$. It thus holds that $\mathsf{coNP^{cc}} (F) = \lceil \log_2 \Rbool(\overline{M}) \rceil$.
  \item A non-deterministic protocol is called {\em unambiguous} if it satisfies that each input has at most one accepting computation. The smallest possible number of communicated bits in such a protocol for $F$ is referred to as the {\em unambiguous non-deterministic communication complexity} of $F$ and is denoted by $\mathsf{UP^{cc}} (F)$.
It holds that $\mathsf{UP^{cc}} (F) = \lceil \log_2 \Rbin(M) \rceil$.
\end{itemize}

\subsection{Composed Functions}\label{sec:compose}

For integers $n$ and $\ell$, let $f : \{0,1\}^n \rightarrow \{0,1\}$ and $g : \{0,1\}^\ell \times \{0,1\}^\ell \rightarrow \{0,1\}$ be two functions.
The function $g^n : \{0,1\}^{\ell \cdot n} \times \{0,1\}^{\ell \cdot n} \rightarrow \{0,1\}^n$ is defined by
\[ g^n (x,y) = (g(x_1,y_1), g(x_2,y_2), \ldots, g(x_n,y_n))\]
for all $x,y \in \{0,1\}^{\ell \cdot n}$, where the vectors $x=(x_1, \ldots, x_n)$ and $y = (y_1, \ldots, y_n)$ are viewed as concatenations of $n$ blocks of size $\ell$.
The {\em composed function} $f \circ g^n : \{0,1\}^{\ell \cdot n} \times \{0,1\}^{\ell \cdot n} \rightarrow \{0,1\}$ is defined by
\[ (f \circ g^n) (x,y) = f( g^n (x,y) ).\]
For a set $I \subseteq [n]$ and a vector $x \in \{0,1\}^{\ell \cdot n}$, we let $x_I \in \{0,1\}^{\ell \cdot |I|}$ denote the projection of $x$ to the blocks whose indices are in $I$. Note that when $I=\{i\}$ for some $i \in [n]$, we have $x_i = x_I$.
For vectors $x,y \in \{0,1\}^{\ell \cdot n}$, we let $g^{I}(x_I,y_I)$ denote the projection of $g^n(x,y)$ to the indices of $I$.

\subsection{Certificate Complexity}\label{sec:query_def}
An $n$-variate {\em $k$-DNF} formula $\varphi$ is a Boolean formula on $n$ variables that can be written as a disjunction $\varphi = c_1 \vee \cdots \vee c_m$, where every $c_i$ is a conjunction of at most $k$ literals.
The formula $\varphi$ is said to be {\em unambiguous} if for every input $x \in \{0,1\}^n$ there is at most one $i \in [m]$ that satisfies $c_i(x)=1$.
For a Boolean function $f: \{0,1\}^n \rightarrow \{0,1\}$, consider the following query complexity measures.
\begin{itemize}
  \item The {\em $1$-certificate complexity} of $f$, denoted by $\mathsf{C}_1(f)$, is the smallest integer $k$ for which $f$ can be written as a $k$-DNF formula.
  \item The {\em $0$-certificate complexity} of $f$, denoted by $\mathsf{C}_0(f)$, is $\mathsf{C}_1(\neg f)$, where $\neg f$ is the negation of $f$. Equivalently, $\mathsf{C}_0(f)$ is the smallest integer $k$ for which $f$ can be written as a $k$-CNF formula.
  \item The {\em unambiguous $1$-certificate complexity} of $f$, denoted by $\mathsf{UC}_1(f)$, is the smallest integer $k$ for which $f$ can be written as an unambiguous $k$-DNF formula.
\end{itemize}

We need the following result that was proved in~\cite{BBGJK21}.

\begin{theorem}[\cite{BBGJK21}]\label{thm:certificate}
For infinitely many integers $r$, there exists a Boolean function $f: \{0,1\}^n \rightarrow \{0,1\}$ satisfying $\mathsf{UC}_1(f) = r$ and $\mathsf{C}_0(f) \geq \widetilde{\Omega}(r^2)$ where $r = n^{\Omega(1)}$.
\end{theorem}

\subsection{Discrepancy}\label{sec:disc}

\begin{definition}[Discrepancy with respect to the uniform distribution]
Let $\Lambda$ be a finite set, and let $g: \Lambda \times \Lambda \rightarrow \{0,1\}$ be a function. Let $X,Y$ be independent random variables that are uniformly distributed over $\Lambda$. The {\em discrepancy} of $g$ (with respect to the uniform distribution) on a combinatorial rectangle $R \subseteq \Lambda \times \Lambda$ is denoted by $\disc_R(g)$ and is defined by
\[ {\disc}_R(g) = \Big | \Prob{}{g(X,Y)=0 \mbox{~and~} (X,Y) \in R} - \Prob{}{g(X,Y)=1 \mbox{~and~} (X,Y) \in R} \Big |.\]
The {\em discrepancy} of $g$, denoted by $\disc (g)$, is defined as the maximum of $\disc_R(g)$ over all combinatorial rectangles $R \subseteq \Lambda \times \Lambda$.
\end{definition}

\section{Lifting from Certificate to Communication Complexity}\label{sec:lifting}

In this section, we present our extension of the query-to-communication lifting theorem in non-deterministic communication complexity to general low-discrepancy functions.
We start with a simple upper bound on the unambiguous non-deterministic communication complexity of a composed function.

\begin{lemma}\label{lemma:lift}
For all functions $f : \{0,1\}^n \rightarrow \{0,1\}$ and $g: \{0,1\}^\ell \times \{0,1\}^\ell \rightarrow \{0,1\}$, it holds that
\[\mathsf{UP^{cc}} ( f \circ g^n) \leq O \big (\mathsf{UC}_1(f) \cdot \max(\log_2 n, \ell) \big ).\]
\end{lemma}

\begin{proof}
Put $k = \mathsf{UC}_1(f)$. Then, the function $f$ can be written as an unambiguous $n$-variate $k$-DNF formula $\varphi = c_1 \vee \cdots \vee c_m$ where $m \leq (2n)^k$.
Consider the following non-deterministic protocol for the communication problem associated with the function $f \circ g^n$.
Let $x,y \in \{0,1\}^{\ell \cdot n}$ be the inputs of the players. The first player selects non-deterministically an index $i \in [m]$ and sends it to the other player.
Let $I \subseteq [n]$ denote the set of indices of the variables that appear in the clause $c_i$, and note that $|I| \leq k$.
Then, the first player sends the projection $x_I$ of $x$ to the blocks of $I$, and similarly, the second player sends the projection $y_I$ of $y$ to the blocks of $I$.
The players accept if and only if $c_i (g^I(x_I,y_I)) = 1$.

Observe that $(f \circ g^n) (x,y) = 1$ if and only if the protocol has an accepting computation on the inputs $x,y$.
Observe further that the fact that $\varphi$ is unambiguous implies that the protocol is unambiguous as well.
Finally, the number of bits communicated by the protocol is
\[O(\log_2 m + k \cdot \ell) \leq O \big ( k \cdot \max (\log_2 n , \ell) \big ),\]
completing the proof.
\end{proof}

We turn to state a lower bound on the co-non-deterministic communication complexity of composed functions $f \circ g^n$ for low-discrepancy functions $g$.
Its proof is given in Appendix~\ref{appendix:lifting}.

\begin{theorem}\label{thm:lift}
For every $\eta >0$ there exists $c>0$ for which the following holds.
Let $\ell$ and $n$ be integers such that $\ell \geq c \cdot \log_2 n$, and let $g: \{0,1\}^\ell \times \{0,1\}^\ell \rightarrow \{0,1\}$ be a function satisfying $\disc(g) \leq 2^{-\eta \cdot \ell}$.
Then, for every function $f : \{0,1\}^n \rightarrow \{0,1\}$, it holds that
\[\mathsf{coNP^{cc}} ( f \circ g^n) \geq \Omega \big ( \eta \cdot \mathsf{C}_0(f) \cdot \ell \big ).\]
\end{theorem}

\section{The Binary and Boolean Rank of Regular Matrices}\label{sec:matrices}

In what follows we consider the notion of strongly unbiased functions and show that compositions with such functions are associated with regular matrices.
We then present a strongly unbiased function and analyze its discrepancy.
Equipped with this function, we apply the lifting theorem from the previous section to prove Theorem~\ref{thm:main}.

\subsection{Strongly Unbiased Functions}\label{sec:unbiased}

Consider the following definition.

\begin{definition}
Let $\ell$ be an integer.
We call a function $g : \{0,1\}^\ell \times \{0,1\}^\ell \rightarrow \{0,1\}$ {\em strongly unbiased } if for every vector $x \in \{0,1\}^\ell$, the number of vectors $y \in \{0,1\}^\ell$ satisfying $g(x,y)=1$ is $2^{\ell-1}$, and for every vector $y \in \{0,1\}^\ell$, the number of vectors $x \in \{0,1\}^\ell$ satisfying $g(x,y)=1$ is $2^{\ell-1}$.
Equivalently, $g$ is strongly unbiased if the matrix associated with $g$ is $2^{\ell-1}$-regular.
\end{definition}

The following lemma shows that compositions with strongly unbiased functions are associated with regular matrices.

\begin{lemma}\label{lemma:regular}
For all functions $g : \{0,1\}^\ell \times \{0,1\}^\ell \rightarrow \{0,1\}$ and $f: \{0,1\}^n \rightarrow \{0,1\}$, if $g$ is strongly unbiased then the matrix associated with the composed function $f \circ g^n$ is regular.
\end{lemma}

\begin{proof}
Let $g : \{0,1\}^\ell \times \{0,1\}^\ell \rightarrow \{0,1\}$ be a strongly unbiased function, let $f: \{0,1\}^n \rightarrow \{0,1\}$ be a function, and let $M$ be the matrix of dimensions $2^{\ell \cdot n} \times 2^{\ell \cdot n}$ associated with the composed function $f \circ g^n$.
Since $g$ is strongly unbiased, it follows that for every vector $x \in \{0,1\}^{\ell \cdot n}$ and for every vector $a \in \{0,1\}^n$, precisely $2^{-n}$ fraction of the vectors $y \in \{0,1\}^{\ell \cdot n}$ satisfy $g^n(x,y)=a$. This implies that the row of the matrix $M$ that corresponds to a vector $x \in \{0,1\}^{\ell \cdot n}$ consists of the evaluations of the function $f$ on all vectors $a \in \{0,1\}^n$, where each such evaluation appears exactly $2^{-n} \cdot 2^{\ell \cdot n} = 2^{(\ell-1)n}$ times. In particular, the number of ones in this row is $2^{(\ell-1)n} \cdot |f^{-1}(1)|$. Since this number is independent of $x$, it follows that this is the number of ones in each row of the matrix $M$. By symmetry, this is also the number of ones in each column of $M$, implying that the matrix $M$ is regular.
\end{proof}

\subsection{The Gadget Function}\label{sec:gadget}

For an integer $\ell \geq 1$, define the function $g_{\ell}: \{0,1\}^\ell \times \{0,1\}^\ell \rightarrow \{0,1\}$ by
\[ g_{\ell}(x,y) = x_1+y_1+ \sum_{i=2}^{\ell}{x_i \cdot y_i}~~(\mod~2)\]
for all $x,y \in \{0,1\}^\ell$. We first observe that $g_\ell$ is strongly unbiased.

\begin{lemma}\label{lemma:g_l_strongly}
For every integer $\ell \geq 1$, the function $g_{\ell}$ is strongly unbiased.
\end{lemma}

\begin{proof}
Consider the function $g_\ell$ for an integer $\ell \geq 1$.
By definition, for every $x \in \{0,1\}^\ell$, it holds that for every $y' \in \{0,1\}^{\ell-1}$ exactly one of the two vectors $y \in \{0,1\}^\ell$ with suffix $y'$ satisfies $g(x,y)=1$. This implies that for every $x \in \{0,1\}^\ell$ precisely $2^{\ell-1}$ of the vectors $y \in \{0,1\}^\ell$ satisfy $g (x,y)=1$. By symmetry, we also have that for every $y \in \{0,1\}^\ell$ precisely $2^{\ell-1}$ of the vectors $x \in \{0,1\}^\ell$ satisfy $g (x,y)=1$, so we are done.
\end{proof}

We turn to show that the functions $g_\ell$ have low discrepancy.
We note that this can be directly derived from a bound on the discrepancy of the inner product function.
Yet, we present below a bound with a somewhat better multiplicative constant, borrowing an argument of Bouda, Pivoluska, and Plesch~\cite{BoudaPP12}.

We start with some definitions.
A {\em Hadamard matrix} is a $\pm1$ matrix in which every two distinct rows and every two distinct columns are orthogonal over the reals.
A standard example for a Hadamard matrix is the $2^\ell \times 2^\ell$ matrix $H_\ell$, with rows and columns indexed by the vectors of $\{0,1\}^\ell$, defined by $(H_\ell)_{x,y} = (-1)^{\sum_{i=1}^{\ell}{x_i \cdot y_i}}$ for all $x,y \in \{0,1\}^\ell$.
A lemma of Lindsey asserts that every submatrix of a Hadamard matrix is quite balanced (for a proof, see, e.g.,~\cite[Lemma~8]{ChorG88}).

\begin{lemma}[Lindsey's Lemma]\label{lemma:lindsey}
Let $H$ be an $n \times n$ Hadamard matrix. Then, the sum of elements in every $r \times s$ submatrix of $H$ is at most $\sqrt{r \cdot s \cdot n}$.
\end{lemma}

\begin{lemma}\label{lemma:disc}
For every integer $\ell \geq 1$, the discrepancy of the function $g_{\ell}$ satisfies $\disc(g_\ell) \leq 2^{-(\ell+3)/2}$.
\end{lemma}

\begin{proof}
Let $M$ denote the $2^\ell \times 2^\ell$ matrix associated with the function $g_\ell$, and let $N$ be the $2^\ell \times 2^\ell$ matrix defined by $N_{i,j} = (-1)^{M_{i,j}}$ for all $i,j \in [2^\ell]$. Observe that the matrix $N$ is equal, up to a permutation of the rows and columns, to the matrix
\[\left(
    \begin{array}{cc}
      H & -H \\
      -H & H \\
    \end{array}
  \right),
\]
where $H = H_{\ell-1}$ is the $2^{\ell-1} \times 2^{\ell-1}$ matrix associated with the inner product function on pairs of vectors of length $\ell-1$.
Let $A,B \subseteq [2^\ell]$ be sets of rows and columns in $N$, and consider the combinatorial rectangle $R = A \times B$. We turn to show that the sum of elements of $N$ in the entries of $R$ does not exceed $2^{3(\ell-1)/2}$.

Observe first that if the set $A$ includes both $i$ and $i+2^{\ell-1}$ for some $i \in [2^{\ell-1}]$, then the sum of the elements of $N$ in the rows of $A \times B$ that correspond to these indices is zero. Letting $A' \subseteq A$ be the set of rows obtained from $A$ by removing those pairs, it suffices to bound the sum of elements of $N$ in the entries of $A' \times B$.
Consider the $2^{\ell-1} \times 2^{\ell}$ submatrix $N'$ of $N$ defined as follows. For every $i \in [2^{\ell-1}]$, if $i \in A'$ then the $i$th row of $N'$ is the $i$th row of $N$, and otherwise it is the $i$th row of $N$ multiplied by $-1$ (i.e., the row of $N$ indexed by $i+2^{\ell-1}$).
Observe that the rectangle $A' \times B$ in $N$ lies in the submatrix $N'$ which can be written as $N' = (H',-H')$, where the $i$th row of $H'$ is either the $i$th row of $H$ or the $i$th row of $H$ multiplied by $-1$. Notice that $H'$ is a Hadamard matrix, and let $A'' \times B$ denote the rectangle in $N'$ that corresponds to the rectangle $A' \times B$ in $N$.

Next, observe that if the set $B$ includes both $i$ and $i+2^{\ell-1}$ for some $i \in [2^{\ell-1}]$, then the sum of the elements of $N'$ in the columns of $A'' \times B$ that correspond to these indices is zero. As before, letting $B' \subseteq B$ be the set of columns obtained from $B$ by removing those pairs, it suffices to bound the sum of elements of $N'$ in the entries of $A'' \times B'$. It now follows that this rectangle lies in a $2^{\ell-1} \times 2^{\ell-1}$ submatrix $H''$ of $N'$, where the $i$th column of $H''$ is either the $i$th column of $H'$ or the $i$th column of $H'$ multiplied by $-1$. Notice that the matrix $H''$ is a Hadamard matrix as well. By Lemma~\ref{lemma:lindsey}, the sum of elements of $H''$ in the entries of $A'' \times B'$ does not exceed $\sqrt{|A''| \cdot |B'| \cdot 2^{\ell-1}}$. By $|A''|,|B'| \leq 2^{\ell-1}$, it follows that the latter is at most $2^{3(\ell-1)/2}$. As explained above, this is also an upper bound on the sum of elements of $N$ in the entries of the rectangle $R$.

Finally, let $m_0$ and $m_1$ denote, respectively, the numbers of zeros and ones of $M$ in the entries of the rectangle $R$. It holds that $|m_0 - m_1| \leq 2^{3(\ell-1)/2}$, and this implies that
\[ {\disc}_R(g_\ell) = \Big |\frac{m_0-m_1}{2^{2\ell}} \Big | \leq \frac{2^{3(\ell-1)/2}}{2^{2\ell}} = 2^{-(\ell+3)/2},\]
completing the proof.
\end{proof}

\subsection{Proof of Theorem~\ref{thm:main}}

We are ready to put everything together and to complete the proof of Theorem~\ref{thm:main}.

\begin{proof}[ of Theorem~\ref{thm:main}]
By Theorem~\ref{thm:certificate}, for infinitely many integers $r$, there exists a Boolean function $f: \{0,1\}^n \rightarrow \{0,1\}$ satisfying $\mathsf{UC}_1(f) = r$ and $\mathsf{C}_0(f) \geq \widetilde{\Omega}(r^2)$ where $r = n^{\Omega(1)}$.
For an integer $\ell$, consider the function $g_\ell : \{0,1\}^\ell \times \{0,1\}^\ell \rightarrow \{0,1\}$ defined in Section~\ref{sec:gadget}.
By Lemma~\ref{lemma:disc}, it holds that $\disc(g_\ell) \leq 2^{-\eta \cdot \ell}$ for $\eta = 1/2$. Theorem~\ref{thm:lift} yields that there exists a constant $c$, such that for $\ell = \lceil c \cdot \log_2 n \rceil$, the composed function $f \circ g_\ell^n$ satisfies
\begin{eqnarray}\label{eq:coNP}
\mathsf{coNP^{cc}} ( f \circ g_\ell^n) \geq \Omega(\mathsf{C}_0(f) \cdot \ell) \geq \widetilde{\Omega}(r^2).
\end{eqnarray}
By Lemma~\ref{lemma:lift}, it further holds that
\begin{equation}\label{eq:UP}
\mathsf{UP^{cc}} ( f \circ g_\ell^n) \leq O (\mathsf{UC}_1(f) \cdot \ell) \leq \widetilde{O}(r),
\end{equation}
where for the second inequality we have used our choice of $\ell$ and the fact that $r = n^{\Omega(1)}$.

To complete the proof, let $M$ be the square $2^{\ell \cdot n} \times 2^{\ell \cdot n}$ matrix associated with the composed function $f \circ g_\ell^n$.
By Lemma~\ref{lemma:g_l_strongly}, the function $g_\ell$ is strongly unbiased, hence by Lemma~\ref{lemma:regular}, the matrix $M$ is regular.
Recalling that $\mathsf{UP^{cc}} ( f \circ g_\ell^n) = \lceil \log_2 \Rbin(M) \rceil$, it follows from~\eqref{eq:UP} that
\begin{eqnarray}\label{eq:RbinM-r}
\Rbin(M) \leq 2^{\widetilde{O}(r)}.
\end{eqnarray}
Put $k = \Rbin(M)$, and combine~\eqref{eq:coNP} and~\eqref{eq:RbinM-r} with the fact that $\mathsf{coNP^{cc}} ( f \circ g_\ell^n) = \lceil \log_2 \Rbool(\overline{M}) \rceil$ to obtain that
\[\Rbool(\overline{M}) \geq 2^{\widetilde{\Omega}(r^2)} \geq k^{\widetilde{\Omega}(\log k)},\]
and we are done.
\end{proof}

\section{The Alon-Saks-Seymour Conjecture and Regular Graphs}\label{sec:graphs}

In this section, we prove the following theorem.

\begin{theorem}\label{thm:M->G}
For every square regular $0,1$ matrix $M$, there exists a simple regular graph $G$ satisfying
\[{\bp}(G) \leq 33 \cdot \Rbin(M)^2 ~~~\mbox{and}~~~ \chi(G) \geq \Rbool(\overline{M})^{1/3}.\]
\end{theorem}
\noindent
As mentioned earlier, it was conjectured by Alon, Saks, and Seymour that every graph $G$ satisfies $\bp(G) \geq \chi(G)-1$, and the conjecture was disproved in a strong sense in a series of works.
Applying Theorem~\ref{thm:M->G} to the matrices given by Theorem~\ref{thm:main} yields {\em regular} graphs that form counterexamples $G$ to the conjecture with a near optimal gap between $\bp(G)$ and $\chi(G)$. This confirms Theorem~\ref{thm:ASS-regular}.

\subsection{Biclique Covering}

We start with some definitions that will be used throughout the proof of Theorem~\ref{thm:M->G}.
All graphs considered here are undirected. They do not contain parallel edges but they may have loops. As usual, a graph is said to be {\em simple} if it contains no loops and no parallel edges.
For a graph $G=(V,E)$, a {\em biclique of $G$} is a complete bipartite subgraph of $G$, that is, a pair $(A,B)$ of sets $A,B \subseteq V$ where every vertex of $A$ is adjacent in $G$ to every vertex of $B$. For adjacent vertices $x,y$ of $G$ such that $x \in A$ and $y \in B$, we say that the biclique $(A,B)$ covers the oriented edge $(x,y)$.
Note that although the edges of $G$ are undirected, a biclique of $G$ covers edges of $G$ with some orientation.
For a set $S \subseteq V$, we let $G[S]$ denote the subgraph of $G$ induced by $S$.

For an integer $t$, a {\em $t$-biclique covering} of $G$ is a collection of bicliques of $G$ that cover every edge of $G$ at least once and at most $t$ times.
The minimum size of such a covering is called the {\em $t$-biclique covering number of $G$} and is denoted by ${\bp}_t(G)$.
For $t=1$, a $1$-biclique covering is also called a {\em biclique partition}, and we write $\bp(G) = \bp_1(G)$.

We need the following result of Bousquet, Lagoutte, and Thomass{\'{e}}~\cite{BousquetLT14}.
For the sake of completeness, we include its short proof in Appendix~\ref{appendix:claim}.
\begin{claim}[{\cite[Claim~28]{BousquetLT14}}]\label{claim:bp_1}
Let $H=(V,E)$ be a simple graph, and let $\calC$ be a $t$-biclique covering of size $k$ of $H$.
Let $E' \subseteq E$ be the set of edges of $H$ that are covered by $\calC$ exactly $t$ times.
Then, the graph $H' = (V,E')$ satisfies $\bp(H') \leq (2k)^t$.
\end{claim}

\subsection{From Regular Matrices to Regular Graphs}

We are ready to prove Theorem~\ref{thm:M->G}.
Before its formal proof, let us briefly describe the proof strategy.
Given a regular $0,1$ matrix $M$, our goal is to construct a simple regular graph $G$ such that its biclique partition number is not much larger than the binary rank of $M$, and its chromatic number is not much smaller than the Boolean rank of $M$. In the first phase of the proof, we use $M$ to construct an intermediate graph $H$. While this graph is not simple, it admits a $2$-biclique covering whose size equals the binary rank of $M$, and it contains a simple subgraph whose chromatic number is at least the Boolean rank of $M$.
In the second phase of the proof, we use $H$ to construct a simple regular graph $G$ with a relatively small $1$-biclique covering and yet a large chromatic number. The regularity of the produced graph $G$ crucially relies on the regularity of the matrix $M$.

\begin{proof}[ of Theorem~\ref{thm:M->G}]
Let $M$ be an $n \times n$ regular $0,1$ matrix, and let $d$ denote the number of ones in each row and each column of $M$.
Put $k = \Rbin(M)$ and $m = \Rbool(\overline{M})$.

We first define a graph $H=(V,E)$ on the vertex set $V = [n] \times [n]$ in which every two (not necessarily distinct) vertices  $(i_1,j_1), (i_2,j_2) \in V$ are adjacent if
\[M_{i_1,j_2}=1~~~\mbox{ or }~~~M_{i_2,j_1}=1.\]
Define $V_0 = \{ (i,j) \in V~|~M_{i,j}=0\}$ and $V_1 = \{ (i,j) \in V~|~M_{i,j}=1\}$.
Note that $V = V_0 \cup V_1$, and notice that the vertices of $H$ that have loops are precisely the vertices of $V_1$.

Let $H_0 = H[V_0]$ denote the subgraph of $H$ induced on the vertices of $V_0$.
Clearly, $H_0$ is a simple graph.
The following lemma relates its chromatic number to the Boolean rank of $\overline{M}$.

\begin{lemma}\label{lemma:chrom_H0}
The graph $H_0$ satisfies $\chi(H_0) \geq m$.
\end{lemma}
\begin{proof}
Put $r = \chi(H_0)$. Then, there exists a partition of $V_0$ into $r$ independent sets $I_1, \ldots, I_r$ of $H_0$.
For each $t \in [r]$, let $A_t$ be the set of elements $i \in [n]$ for which there exists some $j \in [n]$ such that $(i,j) \in I_t$, and let $B_t$ be the set of elements $j \in [n]$ for which there exists some $i \in [n]$ such that $(i,j) \in I_t$. Since $I_t$ is an independent set in $H_0$, it follows that every pair $(i,j) \in A_t \times B_t$ satisfies $M_{i,j}=0$. This implies that $A_t \times B_t$ is a combinatorial rectangle of zeros in the matrix $M$. Since the $r$ given independent sets cover the entire set $V_0$, it follows that for every pair $(i,j) \in V_0$ there exists some $t \in [r]$ such that $(i,j) \in I_t$, and this $t$ satisfies $(i,j) \in A_t \times B_t$.
This shows that the rectangles $A_t \times B_t$ with $t \in [r]$ form a cover of the zeros of $M$, hence $r \geq \Rbool(\overline{M}) = m$, as required.
\end{proof}

The next lemma provides a $2$-biclique covering of $H$ whose size equals the binary rank of $M$.
\begin{lemma}\label{lemma:bp2_H0}
There exists a $2$-biclique covering $\calC$ of $H$ such that
\begin{enumerate}
  \item\label{itm:1} $|\calC| = k$,
  \item\label{itm:2} for every adjacent distinct vertices $(i_1,j_1), (i_2,j_2)$ of $H$, if both $M_{i_1,j_2}=1$ and $M_{i_2,j_1}=1$ hold, then the edge that connects them is covered by $\calC$ twice in the two opposite orientations, and if only one of them holds, then it is covered by $\calC$ once, and
  \item\label{itm:3} every loop of $H$ is covered by $\calC$ once.
\end{enumerate}
\end{lemma}
\begin{proof}
By $k = \Rbin(M)$, there exists a collection of $k$ combinatorial rectangles $A_t \times B_t$ of ones, $t \in [k]$, that forms a partition of the ones of the matrix $M$.
For each $t \in [k]$, define \[C_t = (A_t \times [n], [n] \times B_t),\]
and note that it follows from the definition of $H$ that $C_t$ is a biclique.
Let $\calC$ be the collection of all the bicliques $C_t$ for $t \in [k]$.

Let $(i_1,j_1), (i_2,j_2)$ be two (not necessarily distinct) vertices of $H$.
If $M_{i_1,j_2} =1$ then there exists a unique $t \in [k]$ such that $(i_1,j_2) \in A_t \times B_t$. This implies that the oriented edge $( (i_1,j_1), (i_2,j_2) )$ is covered by the biclique $C_t$ and is not covered by any other biclique of $\calC$. If, however, it holds that $M_{i_1,j_2} =0$, then no $t \in [k]$ satisfies $(i_1,j_2) \in A_t \times B_t$, hence the oriented edge $( (i_1,j_1), (i_2,j_2) )$ is not covered by any biclique of $\calC$.

We turn to show that $\calC$ is a $2$-biclique covering of $H$ that satisfies the assertion of the lemma.
By definition, we have $|\calC|=k$, as required for Item~\ref{itm:1}.
Let $(i_1,j_1), (i_2,j_2)$ be two distinct vertices of $H$.
If the vertices are adjacent then $M_{i_1,j_2}=1$ or $M_{i_2,j_1}=1$.
The above discussion implies that if both the conditions hold then the edge that connects them is covered twice in the two opposite orientations, whereas if only one of the conditions holds, then the edge is covered once, as required for Item~\ref{itm:2}.
For a vertex $(i,j)$ that has a loop, it holds that $M_{i,j}=1$, hence the oriented edge $((i,j), (i,j))$ is covered once by $\calC$, as required for Item~\ref{itm:3}.
On the other hand, if the vertices $(i_1,j_1), (i_2,j_2)$ are not adjacent then $M_{i_1,j_2}=0$ and $M_{i_2,j_1}=0$, hence no oriented edge between them is covered by $\calC$.
It thus follows that $\calC$ is a $2$-biclique covering of $H$, and we are done.
\end{proof}

Let $\calC$ be the $2$-biclique covering of $H$ given by Lemma~\ref{lemma:bp2_H0}.
Consider the two subgraphs of $H_0$ defined by $H_0^{(1)} = (V_0,E_1)$ and $H_0^{(2)} = (V_0,E_2)$, where $E_t$ is the set of edges of $H_0$ that are covered by $\calC$ exactly $t$ times for $t \in [2]$. Notice that the edge set of $H_0$ is $E_1 \cup E_2$.
By assigning to every vertex of $H_0$ the pair of its colors according to some optimal proper colorings of $H_0^{(1)}$ and $H_0^{(2)}$, it follows that
\begin{eqnarray}\label{eq:union_H0}
\chi(H_0) \leq \chi(H_0^{(1)}) \cdot \chi(H_0^{(2)}).
\end{eqnarray}
To obtain the desired simple regular graph, we proceed by considering the following two cases according to the chromatic number of $H_0^{(2)}$.

\paragraph{Case 1.}
Suppose first that $\chi(H_0^{(2)}) \geq m^{1/3}$.
Let $\calC'$ be the collection of bicliques of $H$ obtained from $\calC$ by replacing every biclique $(A,B) \in \calC$ by the three bicliques
\[(A \cap B, A \cap B),~(A \cap B, B \setminus A),~ \mbox{and}~ (A \setminus B, B),\]
where bicliques with an empty part can be avoided.
Observe that these three bicliques cover precisely the same edges covered by $(A,B)$ with the same multiplicities and orientations, where the first biclique has equal parts and the other two have disjoint parts. Note that $A \cap B$ is not necessarily empty because loops are allowed.
It follows that $\calC'$ is a $2$-biclique covering of $H$ of size $|\calC'|\leq 3k$ which satisfies Items~\ref{itm:2} and~\ref{itm:3} of Lemma~\ref{lemma:bp2_H0}.
Letting $\calC'' \subseteq \calC'$ denote the collection of bicliques of $\calC'$ with equal parts, it follows that $|\calC''| \leq k$ and $|\calC' \setminus \calC''| \leq 2k$.

Every biclique of $\calC''$ has the form $(A,A)$ for some set $A \subseteq V$. For every $x \in A$, it covers a loop of $x$ as an oriented edge $(x,x)$, and for every distinct $x,y \in A$, it covers the edge that connects $x$ and $y$ in the two opposite orientations, namely, as $(x,y)$ and as $(y,x)$.
This implies that all the vertices that appear in the bicliques of $\calC''$ have loops in $H$ and thus belong to $V_1$.
Since the parts of the bicliques of $\calC' \setminus \calC''$ are disjoint, it follows that they do not cover any loops, hence the bicliques of $\calC''$ cover all the loops of $H$.
Since $\calC'$ is a $2$-biclique covering of $H$, it follows that no edge is covered by both $\calC''$ and $\calC' \setminus \calC''$.

Let $F$ be the graph obtained from $H$ by removing the edges of the bicliques of $\calC''$.
Since the bicliques of $\calC''$ cover all the loops of $H$, it follows that the graph $F$ is simple.
The collection $\calC' \setminus \calC''$ forms a $2$-biclique covering of $F$, hence ${\bp}_2(F) \leq 2k$.
Let $F^{(2)}$ denote the subgraph of $F$ on $V$ that includes all the edges that are covered by $\calC' \setminus \calC''$ twice.
Since the bicliques of $\calC''$ involve only vertices of $V_1$, it follows that $F^{(2)}$ has an induced subgraph isomorphic to $H_0^{(2)}$, implying that
\begin{eqnarray}\label{eq:ch(F2)}
\chi(F^{(2)}) \geq \chi(H_0^{(2)}) \geq m^{1/3}.
\end{eqnarray}

Now, let $G$ be the graph that contains two disjoint copies of $F^{(2)}$, with additional edges between the two copies according to the bicliques of $\calC''$.
More precisely, $G$ is the graph on the vertex set $V \times [2]$ in which two vertices $(x,b)$ and $(y,b)$ for $b \in [2]$ are adjacent if $x$ and $y$ are adjacent in $F^{(2)}$, and two vertices $(x,1)$ and $(y,2)$ are adjacent if $(x,y)$ is an oriented edge covered by the bicliques of $\calC''$.
The graph $G$ is simple, because $F^{(2)}$ is simple and because no oriented edge is covered twice by $\calC''$.
We claim that $G$ satisfies the assertion of the theorem.

Firstly, $G$ has an induced subgraph isomorphic to $F^{(2)}$, hence it follows from~\eqref{eq:ch(F2)} that
\[\chi(G) \geq \chi(F^{(2)}) \geq m^{1/3}.\]

Secondly, we claim that $\bp(G) \leq 33 \cdot k^2$.
To see this, use Claim~\ref{claim:bp_1} and ${\bp}_2(F) \leq 2k$ to obtain that $\bp(F^{(2)}) \leq (4k)^2$, that is, at most $(4k)^2$ bicliques are needed for a partition of the edges of each copy of $F^{(2)}$ in $G$.
Consider further the bicliques $(A \times \{1\},A \times \{2\})$ for $(A,A) \in \calC''$, which form a partition with size at most $k$ of the edges of $G$ between the vertices of $V \times \{1\}$ and those of $V \times \{2\}$. It follows that
\[\bp(G) \leq 2 \cdot (4k)^2 + k \leq 33 \cdot k^2.\]

Finally, we claim that $G$ is regular with degree $d^2$.
To see this, consider an arbitrary vertex $(i_1,j_1,b) \in V \times [2]$ in $G$.
This vertex is adjacent to the vertices $(i_2,j_2,b)$ for which the pairs $(i_1,j_1)$ and $(i_2,j_2)$ are adjacent in $H$ and the edge that connects them is covered twice by $\calC' \setminus \calC''$.
It is further adjacent to the vertices $(i_2,j_2,b')$ with $b' \neq b$ for which the pairs $(i_1,j_1)$ and $(i_2,j_2)$ are adjacent in $H$ and the edge that connects them is covered by $\calC''$ (twice if they are distinct, and once otherwise).
Since $\calC'$ satisfies Items~\ref{itm:2} and~\ref{itm:3} of Lemma~\ref{lemma:bp2_H0}, it follows that the degree of $(i_1,j_1,b)$ in $G$ is precisely the number of pairs $(i_2,j_2) \in V$ satisfying $M_{i_1,j_2}=1$ and $M_{i_2,j_1}=1$. By the $d$-regularity of $M$, the latter is equal to $d^2$, so we are done.

\paragraph{Case 2.}
Suppose next that $\chi(H_0^{(2)}) < m^{1/3}$.
We start by proving that there exists an independent set $S \subseteq V_0$ in the graph $H_0^{(2)}$ for which
\begin{equation}\label{eq:chiH_0(1)[S]}
\chi(H_0^{(1)}[S]) \geq m^{1/3}.
\end{equation}
Indeed, the assumption implies that there exists a proper coloring of $H_0^{(2)}$ with fewer than $m^{1/3}$ colors. If the induced subgraph of $H_0^{(1)}$ on every color class of this coloring has chromatic number smaller than $m^{1/3}$, then one can obtain a proper coloring of $H_0^{(1)}$ whose number of colors is smaller than $m^{1/3} \cdot m^{1/3} = m^{2/3}$, which implies using~\eqref{eq:union_H0} that $\chi(H_0) < m^{2/3} \cdot m^{1/3} = m$, in contradiction to Lemma~\ref{lemma:chrom_H0}.
This implies that some color class $S \subseteq V_0$ of the coloring of $H_0^{(2)}$ satisfies~\eqref{eq:chiH_0(1)[S]}.

Now, consider the $3$-partite graph $G'$ whose vertex set consists of three copies of $V$ that are connected by three copies of the bicliques of $\calC$ oriented in a cyclic manner. More precisely, the vertex set of $G'$ is $V \times [3]$ and its edges are those of the bicliques
\[ (A \times \{1\}, B \times \{2\}),~~(A \times \{2\}, B \times \{3\}),~~\mbox{and}~~(A \times \{3\}, B \times \{1\}) \]
for all $(A,B) \in \calC$.
By Lemma~\ref{lemma:bp2_H0}, no oriented edge of the bicliques of $\calC$ is covered twice.
It thus follows that $G'$ is a simple graph and that each of its edges is covered by the above bicliques exactly once. By $|\calC| = k$, it follows that $\bp(G') \leq 3k$.
Further, Items~\ref{itm:2} and~\ref{itm:3} of Lemma~\ref{lemma:bp2_H0} imply that the degree of every vertex $(i_1,j_1,b) \in V \times [3]$ of $G'$ is precisely the sum of the number of pairs $(i_2,j_2) \in V$ satisfying $M_{i_1,j_2}=1$ and the number of pairs $(i_2,j_2) \in V$ satisfying $M_{i_2,j_1}=1$. Since the matrix $M$ is $d$-regular, it follows that the graph $G'$ is regular with degree $2nd$.

We next define a graph $G$ as follows.
The graph $G$ is obtained from $G'$ by removing all the edges whose both endpoints are in $S \times [3]$ and by adding the edges of the induced subgraph $H[S]$ of $H$ on $S$ to each of the three copies of $S$ in $G$ (i.e., $S \times \{b\}$ for $b \in [3]$).
Since $G'$ is a simple graph, using the fact that $S$ is a subset of $V_0$ and thus spans no loops, it follows that $G$ is a simple graph as well.
We claim that $G$ satisfies the assertion of the theorem.

Firstly, since $S$ is an independent set in $H_0^{(2)}$, the subgraph of $G$ induced on every copy of $S$ is isomorphic to $H_0^{(1)}[S]$. It thus follows from~\eqref{eq:chiH_0(1)[S]} that
\[\chi(G) \geq \chi(H_0^{(1)}[S]) \geq m^{1/3}.\]

Secondly, we claim that $\bp(G) \leq 9k$.
To see this, recall that $\bp(G') \leq 3k$, and consider some biclique partition with size at most $3k$ of the edges of $G'$.
Replace each biclique $(A \times \{b\}, B \times \{b'\})$ of this partition, where $b \neq b'$, by the two bicliques
\[( (A \setminus S) \times \{b\}, B \times \{b'\})~\mbox{ and }~( (A \cap S) \times \{b\}, (B \setminus S) \times \{b'\}).\]
This gives us a biclique partition with size at most $6k$ of all the edges of $G'$ but those spanned by the vertices of $S \times [3]$.
It remains to cover the edges of the three copies of $H[S]$ in $G$.
Since $S$ is an independent set in $H_0^{(2)}$, each edge of $H[S]$ is covered by $\calC$ exactly once, so by restricting the bicliques of $\calC$ to the vertices of $S$, we get a biclique partition of $H[S]$ with size at most $k$. This gives us a biclique partition with size at most $k$ of the edges of $G[V \times \{b\}]$ for each $b \in [3]$, implying that $\bp(G) \leq 6k+3k = 9k$.

Finally, we claim that $G$ is regular.
To see this, recall that $G'$ is regular and that $G$ is obtained from $G'$ by replacing the edges between the different copies of $S$ by the corresponding edges inside the copies of $S$. Since those edges are covered exactly once by $\calC$, this does not change the degrees of the vertices, yielding that the graph $G$ is regular as well, and we are done.
\end{proof}

\section*{Acknowledgements}
We thank the anonymous reviewers for their helpful and constructive comments.

\bibliographystyle{abbrv}
\bibliography{binary_rank}

\section*{Appendix}
\appendix

\section{Proof of Theorem~\ref{thm:lift}}\label{appendix:lifting}

In this appendix we prove Theorem~\ref{thm:lift}.
We need the following definitions.

\begin{definition}[Min-entropy]
The {\em min-entropy} $\ent(X)$ of a discrete random variable $X$ is defined as
\[\ent (X) = \min_{x \in \supp (X)}{\log_2 \frac{1}{\Prob{}{X=x}}}.\]
Equivalently, $\ent(X)$ is the smallest $b$ for which $\Prob{}{X=x} \leq 2^{-b}$ for every $x$ in the support of $X$.
\end{definition}

\begin{definition}[Density]\label{def:dense}
A pair $(X,Y)$ of random variables over $\{0,1\}^{\ell \cdot n}$ is called {\em $\delta$-dense} if for all sets $I \subseteq [n]$, it holds that $\ent(X_I,Y_I) \geq \delta \cdot 2\ell |I|$.
\end{definition}

We further need the following proposition that was proved in~\cite{ChattopadhyayFK21}.
It says, roughly speaking, that if $g: \{0,1\}^\ell \times \{0,1\}^\ell \rightarrow \{0,1\}$ is a function with low discrepancy and $(X,Y)$ is a pair of independent random variables over $\{0,1\}^{\ell \cdot n}$ whose projection to the blocks of a set $S \subseteq [n]$ is sufficiently dense, then the distribution of $g^S(X_S,Y_S)$ is close to uniform.
A special case of this statement, for $g$ being the inner product function, was previously given in~\cite[Lemma~13]{GoosLMWZ16} (see also~\cite[Lemma~9]{Goos15}).

\begin{proposition}[{\cite[Proposition~3.10]{ChattopadhyayFK21}}]\label{prop:disc}
There exists an absolute constant $h$, such that for every $\eta > 0$ there exists $c>0$ for which the following holds.
Let $\ell$ and $n$ be integers such that $\ell \geq c \cdot \log_2 n$, and let $g: \{0,1\}^\ell \times \{0,1\}^\ell \rightarrow \{0,1\}$ be a function satisfying $\disc(g) \leq 2^{-\eta \cdot \ell}$. For any $\gamma >0$, let $S \subseteq [n]$ be a set, and let $X$ and $Y$ be independent random variables over $\{0,1\}^{\ell \cdot n}$, such that $(X_S,Y_S)$ is $\delta$-dense for $\delta \geq 1+\frac{1}{2}(\gamma-\eta + h/c)$.
Then, for every $a \in \{0,1\}^{|S|}$, it holds that
\[ \Big | \Prob{}{g^S(X_S,Y_S) = a} -2^{-|S|} \Big | \leq 2^{-|S|} \cdot 2^{- \gamma \cdot \ell}.\]
\end{proposition}

Equipped with Proposition~\ref{prop:disc}, we are ready to prove Theorem~\ref{thm:lift}.

\begin{proof}[ of Theorem~\ref{thm:lift}]
Fix $\eta >0$. For some $c>0$ to be determined later, let $\ell$ and $n$ be two integers such that $\ell \geq c \cdot \log_2 n$, and let $g: \{0,1\}^\ell \times \{0,1\}^\ell \rightarrow \{0,1\}$ be a function satisfying $\disc(g) \leq 2^{-\eta \cdot \ell}$. We may and will assume that $n \geq 2$. For a function $f : \{0,1\}^n \rightarrow \{0,1\}$, put $t = \mathsf{coNP^{cc}} ( f \circ g^n)$, and let $M$ denote the $2^{\ell \cdot n} \times 2^{\ell \cdot n}$ matrix associated with $f \circ g^n$. It follows that $\Rbool(\overline{M}) \leq 2^t$, hence there exists a cover $\Pi$ of the zeros of $M$ with at most $2^t$ monochromatic combinatorial rectangles.

Our goal is to show that for some $k \leq O(\frac{t}{\eta \cdot \ell})$ it holds that $\mathsf{C}_0(f) \leq k$, that is, the function $\neg f$ can be represented as a $k$-DNF formula. To do so, it suffices to show that for every $z \in \{0,1\}^n$ satisfying $f(z)=0$, there exists a set $I \subseteq [n]$ of size $|I| \leq k$ such that all vectors $z' \in \{0,1\}^n$ with $z'_I = z_I$ are mapped by $f$ to $0$. Indeed, for every such $z$ and $I$, one can define a conjunction with $|I|$ literals which forms an indicator for the vectors that agree with $z$ on the variables of $I$. The disjunction of all of these conjunctions is an $n$-variate $k$-DNF formula that precisely computes $\neg f$, as required.

Fix a vector $z \in \{0,1\}^n$ satisfying $f(z)=0$. Let $(X,Y)$ be the random variable uniformly distributed over the set
\[ (g^n)^{-1}(z) = \Big \{ (x,y) \in \{0,1\}^{\ell \cdot n} \times \{0,1\}^{\ell \cdot n} \Bigm | g^n(x,y)=z \Big \}.\]
Observe that the random variables $(X_i,Y_i)$ for $i \in [n]$ are independent and that each of them is uniformly distributed over either $g^{(-1)}(0)$ or $g^{(-1)}(1)$.
The assumption $\disc(g) \leq 2^{-\eta \cdot \ell}$ implies that the discrepancy of $g$ on the rectangle $\{0,1\}^{\ell} \times \{0,1\}^{\ell}$ does not exceed $2^{-\eta \cdot \ell}$, hence
\[\big ||g^{-1}(0)| - |g^{-1}(1)| \big | \leq 2^{(2-\eta) \cdot \ell}.\]
This implies that
\[\min \big (|g^{-1}(0)|,|g^{-1}(1)| \big ) \geq 2^{2\ell-1} - 2^{(2-\eta) \cdot \ell-1} \geq 2^{2\ell-2},\]
where the second inequality holds for $\ell \geq c \cdot \log_2 n$ assuming that $c \geq 1/\eta$.
It thus follows that for every set $I \subseteq [n]$, it holds that
\begin{eqnarray}\label{eq:ent_XIYI}
\ent(X_I,Y_I)  = \sum_{i \in I}{\ent(X_i,Y_i)} \geq |I| \cdot \log_2 (2^{2\ell-2}) = |I| \cdot (2\ell-2).
\end{eqnarray}

By $f(z)=0$, the entries of $(g^n)^{-1}(z)$ in $M$ are all zeros. Since $\Pi$ is a cover of the zeros in $M$ with at most $2^t$ rectangles, there must exist a rectangle $R \in \Pi$ that covers at least $2^{-t}$ fraction of the entries of $(g^n)^{-1}(z)$. Let $(X',Y')$ be the random variable uniformly distributed over $(g^n)^{-1}(z) \cap R$.
Note that for every $I \subseteq [n]$, the random variable $(X'_I,Y'_I)$ is obtained from $(X_I,Y_I)$ by conditioning it on the event $(X,Y) \in R$, whose probability is at least $2^{-t}$.
It thus follows, using~\eqref{eq:ent_XIYI}, that for every $I \subseteq [n]$,
\begin{eqnarray}\label{eq:ent_X'IY'I}
\ent(X'_I,Y'_I) \geq \ent(X_I,Y_I) -t \geq |I| \cdot (2\ell-2)-t.
\end{eqnarray}

The following lemma shows that by fixing relatively few blocks in $(X',Y')$, one can get a random variable that is quite dense on the remaining blocks (recall Definition~\ref{def:dense}).
\begin{lemma}\label{lemma:dense}
For every $\delta < 1-\frac{1}{\ell}$, there exist a set $I \subseteq [n]$ of size $|I| \leq \frac{t}{2\cdot ((1-\delta)\ell-1)}$ and an assignment $\alpha \in \{0,1\}^{2 \cdot \ell |I|}$ for which the random variable $(X'',Y'')$ obtained from $(X',Y')$ by conditioning it on the event $(X'_I,Y'_I)=\alpha$ satisfies that its projection $(X''_{\overline{I}},Y''_{\overline{I}})$ to the blocks of $\overline{I} = [n] \setminus I$ is $\delta$-dense.
In addition, letting $X'''$ and $Y'''$ be independent copies of $X''$ and $Y''$ respectively, the random variable $(X'''_{\overline{I}},Y'''_{\overline{I}})$ is $(2\delta-1)$-dense.
\end{lemma}
\begin{proof}
Fix an arbitrary $\delta < 1-\frac{1}{\ell}$.
If the random variable $(X',Y')$ is $\delta$-dense, then the choice $I = \emptyset$ clearly satisfies the assertion of the first part of the lemma.
Otherwise, $(X',Y')$ is not $\delta$-dense, so there exists a set $I \subseteq [n]$ for which $\ent(X'_I,Y'_I) < \delta \cdot 2\ell |I|$. Let $I$ be such a set with maximum size. By~\eqref{eq:ent_X'IY'I}, we obtain that
\[ |I| \cdot (2\ell-2)-t \leq \ent(X'_I,Y'_I) < \delta \cdot 2 \ell |I|,\]
which implies, using $\delta < 1-\frac{1}{\ell}$, that $|I| \leq \frac{t}{2\cdot ((1-\delta)\ell-1)}$.

It follows from $\ent(X'_I,Y'_I) < \delta \cdot 2 \ell |I|$ that there exists an $\alpha \in \{0,1\}^{2 \cdot \ell |I|}$ for which the probability that $(X'_I,Y'_I)=\alpha$ is larger than $2^{-\delta \cdot 2\ell |I|}$.
Let $(X'',Y'')$ be the random variable obtained from $(X',Y')$ by conditioning it on the event $(X'_I,Y'_I)=\alpha$.
We claim that its projection $(X''_{\overline{I}},Y''_{\overline{I}})$ to the blocks of $\overline{I}$ is $\delta$-dense.
To see this, suppose in contradiction that there exists a non-empty set $J \subseteq \overline{I}$ and an assignment $\beta \in \{0,1\}^{2 \cdot \ell |J|}$ for which the probability that $(X''_J,Y''_J)=\beta$ is larger than $2^{-\delta \cdot 2\ell |J|}$. It thus follows that the probability that $(X'_I,Y'_I) = \alpha$ and $(X'_J,Y'_J) = \beta$ is larger than $2^{-\delta \cdot 2\ell |I|} \cdot 2^{-\delta \cdot 2\ell |J|} = 2^{-\delta \cdot 2\ell |I \cup J|}$, hence the set $I \cup J$ violates the $\delta$-density of $(X',Y')$ and contradicts the maximality of $I$.

Now, let $X'''$ and $Y'''$ be independent copies of $X''$ and $Y''$ respectively.
We turn to show that the random variable $(X'''_{\overline{I}},Y'''_{\overline{I}})$ is $(2\delta-1)$-dense.
To see this, fix any $J \subseteq \overline{I}$, and observe that
\[\ent(X'''_{J}) \geq \ent(X''_J,Y''_{J}) - \ent(Y''_J) \geq \delta \cdot 2\ell |J| - \ell |J| = (2\delta-1) \cdot \ell |J|.\]
Similarly, we have $\ent(Y'''_{J}) \geq (2\delta-1) \cdot \ell |J|$. We derive that
\[\ent(X'''_{J},Y'''_{J}) = \ent(X'''_{J}) + \ent(Y'''_{J}) \geq (2\delta-1) \cdot 2\ell |J|,\]
which implies that $(X'''_{\overline{I}},Y'''_{\overline{I}})$ is $(2\delta-1)$-dense, as desired.
\end{proof}

We turn to apply Proposition~\ref{prop:disc}.
Put $\gamma = 1/\ell$.
For the given $\eta > 0$, define
\[\delta = 1+\frac{1}{4} \cdot \Big (\gamma-\eta+\frac{h}{c} \Big ),\]
where $h$ is the constant given in the proposition.
The assumption $\ell \geq c  \cdot \log_2 n \geq c$ implies, for a sufficiently large $c$, say $c > \max \big ( 2 \cdot (h+1),9 \big ) \cdot \eta^{-1}$, that
\begin{eqnarray}\label{eq:delta}
\delta \leq 1+ \frac{1}{4} \cdot \Big ( -\eta+\frac{h+1}{c} \Big ) < 1 - \frac{\eta}{8} < 1- \frac{1}{c} \leq 1- \frac{1}{\ell}.
\end{eqnarray}
By~\eqref{eq:delta}, we can apply Lemma~\ref{lemma:dense} with the above $\delta$.
Let $I \subseteq [n]$ and $\alpha \in \{0,1\}^{2 \cdot \ell |I|}$ be the set and assignment given by the lemma for this $\delta$, and let $(X'',Y'')$ and $(X''',Y''')$ be the corresponding random variables.
Using the inequality $\delta < 1- \frac{\eta}{8}$ that follows from~\eqref{eq:delta}, we obtain from Lemma~\ref{lemma:dense} that $|I| \leq O(\frac{t}{\eta \cdot \ell})$ and that the random variable $(X'''_{\overline{I}},Y'''_{\overline{I}})$ is $(2\delta-1)$-dense. Notice that
\[ 2\delta-1 = 1+\frac{1}{2} \cdot \Big (\gamma-\eta+\frac{h}{c} \Big ). \]
This allows us to apply Proposition~\ref{prop:disc} with the set $S = \overline{I}$ and to obtain, assuming that $c = c(\eta)$ is sufficiently large, that for every $a \in \{0,1\}^{|S|}$,
\[ \Big | \Prob{}{g^S(X'''_S,Y'''_S) = a} -2^{-|S|} \Big | \leq 2^{-(|S|+1)}.\]
This in particular yields that the random variable $g^n(X''',Y''')$ has full support on the entries of $\overline{I}$.

It remains to show that for every $z' \in \{0,1\}^n$ that satisfies $z'_I = z_I$, it holds that $f(z')=0$.
Let $R'$ be the rectangle of the matrix $M$ whose rows and columns are the supports of $X'''$ and $Y'''$ respectively.
Since the rows and columns of $R'$ are also rows and columns of $R$, it follows that $R' \subseteq R$, hence all of its pairs are mapped by $f \circ g^n$ to zero.
By construction, the pairs $(x,y) \in R'$ satisfy $(x_I,y_I) = \alpha$, and it holds that $z \in g^n(R')$.
Since the random variable $g^n(X''',Y''')$ has full support on the entries of $\overline{I}$, it follows that for every vector $z' \in \{0,1\}^n$ with $z'_I = z_I$, there exists a pair $(x,y) \in R'$ such that $g^n(x,y)=z'$.
Since the pairs of $R'$ are mapped by $f \circ g^n$ to zero, we get that $(f \circ g^n) (x,y) = f(g^n (x,y)) = f(z') = 0$, and we are done.
\end{proof}

\section{Proof of Claim~\ref{claim:bp_1}}\label{appendix:claim}

\begin{proof}[ of Claim~\ref{claim:bp_1}]
Let $(A_1,B_1), \ldots, (A_k,B_k)$ be the $k$ bicliques of the $t$-biclique covering $\calC$ of $H$.
By definition, every edge of $E'$ is covered by exactly $t$ of the bicliques of $\calC$.
Consider the function that maps every such edge  $e = \{u,v\} \in E'$ to a label $L = (i_1,\ldots,i_t,P)$,
where $i_1<\cdots<i_t$ are the $t$ indices $i \in [k]$ for which the biclique $(A_{i},B_{i})$ covers the edge $e$, and $P = \{P_1,P_2\}$ is a partition of $[t]$ defined by $P_1 = \{j \in [t] \mid u \in A_{i_j}\}$ and $P_2 = \{j \in [t] \mid u \in B_{i_j}\}$. Note that the partition $P$ can be equivalently defined using the vertex $v$ rather than $u$.

We claim that for every label $L$, the edges of $E'$ that are mapped to $L$ form a biclique in $H'$, and that these bicliques are edge-disjoint.
To see this, suppose that two edges $\{u,v\},\{u',v'\} \in E'$ are mapped to the same label $L=(i_1,\ldots,i_t,P)$.
Then, the two edges are covered by all the bicliques $(A_{i_j},B_{i_j})$ with $j \in [t]$, and it can be assumed, without loss of generality, that $u$ and $u'$ belong to the same part in each of them.
This implies that these bicliques also cover the edges $\{u,v'\}$ and $\{u',v\}$.
Since $\calC$ is a $t$-biclique covering of $H$, it follows that these edges belong to $E'$ and are also mapped to the label $L$.
This implies that the edges of $E'$ that are mapped to $L$ form a biclique in $H'$.
Since the label of every edge in $E'$ is uniquely defined, every such edge is covered by exactly one of these bicliques.
It thus follows that the collection of bicliques associated with all possible labels forms a biclique partition of $H'$.
Since the number of labels is at most $(2k)^t$, it follows that $\bp(H') \leq (2k)^t$, as desired.
\end{proof}

\end{document}